\documentclass[11pt]{article}
\usepackage{amsmath,amssymb, amsthm,delarray}
\usepackage{graphicx}
\numberwithin{equation}{section}
\newtheorem{thm}{Theorem}[section]
\newtheorem{lemma}[thm]{Lemma}

{\rm}
{\rm}
{\rm}

\def\beq{\begin{equation} }
\def\eeq{\end{equation} }

\def\N{\mathbb{N}}
\def\R{\mathbb{R}}

\def\P{\mathbf{P}}
\def\K{\mathbf{K}}

\def\M{\mathbf{M}}

\def\H{\mathcal{H}}

\def\x{\mathbf{x}}
\def\y{\mathbf{y}}

\def\s{\mathbb{S}}
\def\p{\mathcal{P}}
\def\r{\mathcal{R}}
\def\L{\mathcal{L}}

\def\la{\langle}
\def\ra{\rangle}

\def\R{\mathbb{R}}

\def\rd{\mathbb{R}[\mathbf{x}]_{2d}}

\begin{document}
\title{The truncated K-moment problem for closure of open sets\thanks{Corrections to the version in J. Functional Analysis {\bf 263} (2012), pp. 3604--3616, where
we fix a mistake in the proof of Lemma 2.2.}}
\author{G. Blekherman\\
{\scriptsize School of Mathematics
Georgia Institute of Technology}\\
{\scriptsize 686 Cherry Street, Atlanta, GA 30332, USA}\\
{\scriptsize {\tt greg@math.gatech.edu}}\and
J.B. Lasserre\\
{\scriptsize LAAS and Institute of Mathematics, University of Toulouse}\\
{\scriptsize LAAS, 7 avenue du Colonel Roche,
 31077 Toulouse cedex 4, France}\\
 {\scriptsize {\tt lasserre@laas.fr}}}
 
\date{}
 \maketitle

\begin{abstract}
We consider the truncated $\K$-moment problem when
$\K\subseteq\R^n$ is the closure of a, not necessarily bounded, open set
(which includes the important cases $\K=\R^n$ and $\K=\R^n_+$).
We completely characterize the interior of
the convex cone of finite sequences that have a representing measure
on $\K\subseteq\R^n$.
It is in fact the domain of the Legendre-Fenchel transform associated with
a certain convex function. And so in this context, detecting whether a sequence is in the interior of this cone
reduces to solving a finite-dimensional convex optimization problem. This latter problem is related to
{\it maximum entropy} methods for approximating an unknown density
from knowing only finitely many of its moments. Interestingly, the proposed approach is essentially geometric and
of independent interest, as it also addresses the abstract problem of characterizing the interior of a convex cone $C$
which is the conical hull of a set continuously parametrized by a compact set $\M\subset\R^n$ or $\M \subseteq \s^{n-1}$, where $\M$ is the closure of an open subset of $\R^n$ (resp. $\s^{n-1}$). As a by-product we also obtain a barrier function for the cone $C$.
{\bf Subject class:} 44A60 65K10 42B10
\end{abstract}

\section{Introduction}

We are concerned with the (real) truncated $\K$-moment problem, that is, given a closed set $\K\subseteq\R^n$,
and a {\it finite} sequence $\y=(y_\alpha)$, $\alpha\in\N^n_{2d}$ (where $\N^n_{2d}=\{\alpha\in\N^n: \sum_i\alpha_i\leq 2d\}$),
provide conditions under which $\y$ has a representing Borel measure on $\K$, i.e., $\y$ is such that
\[y_\alpha\,=\,\int_\K x^\alpha\,d\mu(x),\qquad\forall\alpha\in\N^n_{2d},\]
for some finite Borel measure $\mu$ on $\K$.

{\bf Background:} For the one-dimensional (or univariate) case $n=1$, this classical problem is well understood and dates back
to contributions by  famous mathematicians, among them
Markov, Stieltjes, Hausdorff, and Hamburger, at the end of the nineteen and beginning of twentieth centuries. Explicit
conditions on the sequence $\y$ exist, all stated in terms of positive semidefiniteness of some Hankel matrices
whose entries are linear in the variables $\y$ (see e.g.
Curto and Fialkow \cite{curto0}); in modern language, these conditions are Linear Matrix Inequalities
(in short LMIs) in $\y$.

For the multi-dimensional ($n>1$) case, no such strong results exist, even for the full moment problem,
and for instance, the full moment problem with $\K=\R^n$ is still unsolved. The Riesz-Haviland criterion states that an infinite sequence
$\y=(y_\alpha)$, $\alpha\in\N^n$,  has a representing measure on $\K$ if and only if $\y$ (viewed as a linear functional acting on the polynomials)
is nonnegative for all polynomials nonnegative on $\K$;  but since there is no tractable characterization of
the latter polynomials, the Riesz-Haviland criterion is not practical. Existence of a representing measure
is related to existence of commuting self-adjoint extensions of (multiplication) operators on polynomials,
defined from the sequence $\y$ (see e.g. Berg \cite{berg}, Sarason \cite{sarason}, Simon \cite{simon}, Vasilescu \cite{vasilescu}) and so far, the most powerful (and general) result is due to Schm\"udgen \cite{schmudgen}, who solved the full $\K$-moment
problem when $\K$ is a compact basic semi-algebraic set of the form
$\K:=\{x\in\R^n: g_j(x)\geq0, j=1,\ldots,m\}$ for some polynomials $(g_j)\subset\R[\x]$. In this context, a
sequence $\y$ has a representing measure on $\K$ if and only if it satisfies countably many (explicit) LMI's;
this result was later refined (and simplified) by Putinar \cite{putinar}
when the quadratic module generated by the $g_j$'s  is Archimedean.
Later, the full $\K$-moment problem for basic closed (not necessarily compact) semi-algebraic sets
was also solved (at the price of a dimensional extension) in Putinar and Vasilescu \cite{putval}.
Finally, there also exist
conditions in terms of linear inequalities on $\y$, based on an alternative
representation theorem initially due to Krivine \cite{krivine1,krivine2},
and also later in Marshall \cite{marshall} and Vasilescu \cite{vasilescu}.

However, for the truncated moment problem in a general context, the only ``explicit" criterion is the so-called
{\it flat extension} of positive moment matrices in Curto and Fialkow \cite{curto1,curto2}.
(A positive semidefinite moment matrix $\M_d(\y)$ associated with $\y\in\N^n_{2d}$ has a flat extension
if the sequence $\y$ can be extended to $\tilde{\y}\in\N^n_{2d+2}$ in such a manner that
the resulting moment matrix $\M_{d+1}(\tilde{\y})$ has same rank as $\M_d(\y)$.)
Namely, $\y=(\y_\alpha)$, $\alpha\in\N^n_{2d}$, has a representing measure on $\R^n$
if $\M_d(\y)$ is positive semidefinite and $\y$ can be extended to $\tilde{\y}\in\N^n_{2(d+k)}$ for some $k$,
in such a manner that $\tilde{\y}$ has a flat extension.
But again this test is not practical.  Finally, in Jordan and Wainwright \cite{jordan}
the authors characterize the sequences $\y$ that have a representing measure
with a density with respect to a reference measure.

{\bf Contribution:} We consider the truncated $\K$-moment problem where $\K\subseteq\R^n$
is the closure of a not necessarily bounded open set;  and so, in particular,
it includes the important special cases $\K=\R^n$ and $\K=\R^n_+$. In this context, 
we completely characterize the interior of the convex cone
$C(\K)\subset\R^{s(n,2d)}$ (where $s(n,d):=\binom{n+d}{n}$) of finite sequences $\y=(y_\alpha)$, $\alpha\in\N^n_{2d}$,
that have a finite representing measure on $\K$.

Namely, let
$\mu$ be any measure on $\K$, absolutely continuous with respect to the Lebesgue measure on $\K$.
Then, if $\K$ is compact we show that any sequence
$\y\in{\rm int}(C(\K))$ has a representing measure $\nu$ absolutely continuous with respect to $\mu$, and such that
\[y_\alpha\,=\,\int_\K x^\alpha\,\underbrace{e^{p(s)}\,d\mu(s)}_{d\nu(s)},\qquad\forall \alpha\in\N^n_{2d},\]
for some polynomial $p\in\R[x]$ of degree at most $2d$.

We note that the above criterion depends only on the moments of degree at most $2d$, and does not rely on extending the moment matrix or taking higher moments into account. 

If $\K=\R^n$ then we take a little detour in $\R^{n+1}$ by homogenization, so that
$\mu$ is now a rotation invariant measure on the unit sphere $\s^n\subset\R^{n+1}$, and
$\y\in {\rm int}(C(\K))\subset\R^{s(n+1,2d)}$ has a representing measure $\nu$ absolutely continuous with respect to $\mu$, and such that
\begin{equation}
\label{notrn}
y_\alpha\,=\,\int_{\s^n} x^\alpha\,\underbrace{e^{p(s)}\,d\mu(s)}_{d\nu(s)},\qquad\forall \alpha\in\N^{n+1}_{2d},\end{equation}
for some homogeneous polynomial $p\in\R[x_0,\cdots ,x_{n}]$ of degree $2d$.
Importantly, the above result holds, with identical proof via homogenization, when $\K$ is the closure of an unbounded open subset
of $\R^n$ (like e.g. $\R^n_+$). In this case, and as explained and detailed in Section \ref{SUBSECTION Non-Compact},
integration on  $\s^n$ in (\ref{notrn}) is replaced with integration on some compact subset $\M$ of $\s^n$.

Alternatively, $\y\in {\rm int}(C(\K))$ has a representing measure if and only
if $f^*(\y)<+\infty$, where $f^*:\R^{s(n,2d)}\to\R\cup\{+\infty\}$ is the Legendre-Fenchel transform of the convex function
\begin{equation}
\label{0}
p\mapsto f(p)\,:=\, \int_\K\,e^{p(s)}\,d\mu(s)\quad\left(\mbox{or }\:\int_{\M}\,e^{p(s)}\,d\mu(s)\right),\end{equation}
defined for polynomials of degree at most $2d$. That is, $f^*$ is defined as
\[\y\mapsto f^*(\y)\,:=\,\sup_{p \in \rd} \:\{p^T\y-f(p)\:\}.\]
And so, checking whether $\y\in{\rm int}\,(C(\K))$
reduces to solving the finite-dimensional convex optimization problem $\P$ of finding the supremum of $p^T\y-f(p)$. We show that the supremum is finite and \textit{attained} on the interior of the cone $C(\K)$ and $f^*(\y)=+\infty$ for $\y$ not in the interior. This means that $f^*$ (resp. $\log f^*$) provides a {\it barrier} (resp. log-barrier) function for the cone $C(\K)$.

Our result is in the vein of (and extends)
Wainwright and Jordan \cite[Theorem 3.3]{jordan}, where (when the domain of $f$ is open)
the authors have shown that  the gradient
map $\nabla \log f$ is {\it onto} the interior of the convex set $\mathcal{M}$ of sequences $\y$ that have total mass $1$ and a
representing measure {\it with a density} with respect to a reference measure $\mu$; here we
prove the same result for the interior of the (potentially larger) convex cone of  sequences that have a representing measure not necessarily absolutely continuous
w.r.t. $\mu$. Furthermore, we guarantee that the Legendre-Fenchel transform $f^*$ is finite 
only on the interior of the cone of representable moment sequences, while \cite[Theorem 3.3]{jordan} makes no guarantees of the behavior of $(\log f)^*$ on the boundary. 
This extra regularity of $f^*$ allows us to conclude that it is a barrier function for the cone $C(\mathbf{K})$. 

However, and even though $\P$ is a finite dimensional convex problem,
effective numerical computation of $f^*(\y)$ is still difficult. This is because
evaluating $f$ and its gradient $\nabla f$ (as well as its Hessian $\nabla^2f$ for second-order methods)
at a point $y$, requires evaluating integrals over $\K$, a difficult problem. However, notice that  if $\K$ is relatively ``simple", 
one may approximate those integrals by using
cubature formulas, or discretization schemes, or Monte-Carlo methods; also, 
for small dimensions $2$ and $3$, and if $\K$ is defined by polynomials, then these integrals
can be approximated efficiently and accurately by dimension reduction to line and surface integrals; see e.g.
Wester et al. \cite{integration}.

The optimization problem $\P$ is well-known and called the {\it maximum-entropy} approach
(in our case, the Boltzmann-Shannon entropy) for estimating an unknown density from the only knowledge of finitely many of its moments. In maximum entropy the sequence $\y$ is {\it known} to come from a representing measure
and the main question of interest is the convergence of an optimal solution
$p^*_d\in \R[x]_{2d}$ of $\P$ when the number of moments (i.e., $d$) increases.
For a detailed account of such results, the interested reader is referred to Borwein and Lewis \cite{borwein}
and the many references therein. And so, another contribution of this paper is to show that 
the maximum entropy approach not only permits to approximate an unknown density but also
permits to solve the $\K$-moment problem for closure of open sets.

Finally, our (essentially geometric) approach is also of  independent interest, as we obtain our result on the
truncated $\K$-moment problem as a by-product of
the more abstract problem of characterizing the interior of a convex cone
$C\subset\R^m$ which is the conical hull of a compact set $\L(\M)$ continuously parametrized by a set $\M\subset\R^n$ or $\M\subseteq\s^n$, where $\M$ is the closure of an open subset of $\R^n$ (resp. $\s^{n-1}$).

\section{Notation, definitions and preliminary results}
\label{prelim}
Let $V$ be the Euclidean space with an inner product $\la,\ra$ and
$\L:\R^n \to V$ a continuous mapping. Given a compact set $\M\subset\R^n$  that is the closure of a bounded open subset of $\R^n$ or $\s^{n-1}$,
let $\L(\M)\subset V$ be a compact subset lying in an affine hyperplane $H\subset V$.

Let $h \in V$ be the vector perpendicular to $H$ such that $\la x,h \ra=1$ for all $x \in H$.
Let $\mu$ be a measure on $\M$ absolutely continuous with respect to the Lebesque measure or the rotation-invariant measure on $\s^{n-1}$, and with a density positive on $\M$, and normalize $\mu$ to have mass $1$.

We are primarily interested in the conical hull of $\L(\M)$, which is the convex cone
\[C\,:=\,\operatorname{ConicalHull}(\L(\M))\,=\,\left\{\sum \lambda_ix_i \, | \, \lambda_i\geq 0, \, x_i \in \L(\M)\right\}.\]
Without loss of generality we may assume that the affine hull of $\L(\M)$ is all of $H$, so that the cone $C$ is full-dimensional in $V$.  Define the function $f:V\rightarrow \R$ as follows:
\begin{equation}
\label{def-f}
x\mapsto f(x)\,:=\,\int_\M e^{\la x,\L(v)\ra}\, d\mu(v),\qquad x\in V.\end{equation}

The function $f$ is smooth and strictly convex (follows from the definition of $\mu$ and
$\M$ compact). Using the compactness of $\M$ we may differentiate under the integral sign to obtain:
\begin{equation*}
\nabla f(x)\,=\,\int_\M \L(v) \,e^{\la x,\L(v)\ra}\, d\mu(v),\qquad x\in V.
\end{equation*}

It follows that $f$ is a function of ``Legendre-type" in the sense of Rockafellar \cite[Chapter 26, p. 258]{roc}. If we think of $\nabla f$ as a function mapping $V$ to $V$, then it is clear that the image of $\nabla f$ lies in $C$. Define the Legendre-Fenchel transform $f^*:V \rightarrow \R$ of $f$ as follows:
\begin{equation*}
y\mapsto f^*(y)\,:=\,\sup_{x \in V} \, \la x,y \ra -f(x),\qquad y\in V.
\end{equation*}
By \cite[Theorem 26.5]{roc}, $\nabla f$ is one-to-one and the image of $\nabla f$ is the interior of the domain of $f^*$
(i.e., points where $f^*$ is finite). In particular, it follows that the image of $\nabla f$ is convex. Moreover, the inverse of $\nabla f$, viewed as a mapping from $V$ to $V$, is just $\nabla f^*$, i.e.:
\[(\nabla f)^{-1}=\nabla f^*.\]

Next, we introduce a couple of intermediate lemmas that we will need later to prove our main result. They follow from elementary
convexity and analysis and for clarity of exposition, their proofs are postponed until Section \ref{SECTION Lemma Proofs}.

\begin{lemma}\label{LEMMA Convex}
Let $B$ be a compact convex set and let $A$ be a convex subset of $B$ such that the closure $\overline{A}$ of $A$ contains all exposed extreme points of $B$. Then $A$ contains the interior of $B$.
\end{lemma}

\begin{lemma}\label{LEMMA Dirac}
Let $\M$ be a compact set that is the closure of an open subset of $\R^n$ or $\s^{n-1}$, and let $\mu$ be a measure on $\M$, absolutely continuous with respect to the Lebesque or rotation invariant measure on $\s^{n-1}$, and with a density positive on $\M$. Suppose that $f: \M \rightarrow \R$ is a continuous function such that $f$ is positive on $\M$ and it attains its maximum at a unique point $s \in \M$. Then for all continuous functions $g:\M \rightarrow \R$,
\begin{equation}
\label{result}
\lim_{\lambda \rightarrow \infty}\frac{\displaystyle\int_\M gf^{\lambda}\, d\mu}{\displaystyle\int_\M f^{\lambda}\, d\mu}=g(s).\end{equation}
\end{lemma}

We now show that the image of $\nabla f$ is the interior of $C$.

\begin{thm}\label{THEOREM Gradient Image}
The image of $\nabla f$ is the interior of $C$.
\end{thm}
\begin{proof}
From the above it follows that the image of $\nabla f$ is an open convex set. Observe that for all $a \in \R$
\[f(x+ah)=\int_\M e^{\la x+ah,\L(v)\ra}\, d\mu(v)=e^a\int_\M e^{\la x,\L(v)\ra}\, d\mu(v),\qquad \forall x\in V.\]
Therefore,
\[\nabla f(x+ah)=e^a\nabla f(x),\qquad \forall a\in\R,\,\forall x\in V.\]
It follows that the image of $\nabla f$ is an open convex sub-cone of $C$.

Next, let $B=\operatorname{conv} \L(\M)$ (the convex hull of $\L(\M)$)  be the {\it base} of the cone $C$. It suffices to show that $\nabla f$ is onto the interior of $B$. By Lemma \ref{LEMMA Convex} it is enough to show that we can approximate any exposed extreme point of $B$ arbitrarily well by points $\nabla f(x_i)$ for some sequence $\{x_i\}\subset V$.

So, let $s \in B$ be an exposed extreme point. It follows that $s \in \L(\M)$ and we can write $s=\L(s')$ with $s' \in \M$. Furthermore there exists $p \in V$ such that $\la p,\L(s') \ra=0$ and $\la p,\L(v)\ra <0$ for all $v \in \M$ with $v \neq s'$. Now consider the point $\nabla f(\beta \,p)$, $\beta\in\R_+$, i.e.,
\begin{equation*}
\nabla f(\beta \,p)=\int_{\M} \L(v)\, e^{\beta\, \la p,\L(v)\ra}\, d\mu(v),\qquad\beta\in\R_+.
\end{equation*}
To make sure that the points we consider lie in $B$ we need to divide by $\int_\M e^{\beta\,\la p,\L(v)\ra}\, d\mu(v).$ Define
\[s_{\beta}\,:=\,\frac{\displaystyle\int_\M \L(v) \,e^{\beta\,\la p,\L(v)\ra}\, d\mu(v)}{\displaystyle\int_\M e^{\beta\,\la p,\L(v)\ra}\, d\mu(v)}.\]

Applying Lemma \ref{LEMMA Dirac} with $f:=e^{\la p,\L(v) \ra}$ yields $\lim_{\beta \rightarrow \infty}s_{\beta}=s$.
Therefore, the sequence of points $\nabla f(\beta\,p+a\,h)$, $\beta\in\R_+$, with
$e^{-a}:=\int_\M e^{\beta\,\langle p,\L(v)\rangle}d\mu(v)$, approximate $s$, the desired result.
\end{proof}

We know from \cite{roc} that the image of $\nabla f$ is the interior of the domain of $f^*$. We have shown that on the interior of the image of $\nabla f$ the supremum is always attained. Also, outside of $C$ we know that $f^*$ is equal to $+\infty$. Now we show that under an additional assumption on the geometry of the embedding $\L(\M)$, $f^*=+\infty$ on the boundary of $C$. We will show that the additional assumption holds in our applications of interest.

\begin{lemma}\label{LEMMA Assumption}
Suppose that for any face $F$ of $C$ we have $\mu(\{v \,\mid \, \L(v) \in F\})=0$. Then $f^*(y)=+\infty$ for all $y$ in the boundary of $C$.
\end{lemma}
\begin{proof}
It suffices that show that $f^*(y)=+\infty$ for all $y$ in the boundary $\partial B$ of the base $B$. Let $y \in \partial B$. Then there exists $p \in V$ such that $\la p,y \ra =0$ and $\la p,\L(v) \ra \leq 0$ for all $v \in \M$. Consider
\[m_{\alpha,\beta}=\sup_{\alpha,\beta} \,\, \la \alpha h +\beta p,y \ra - f(\alpha h +\beta p).\]
It follows that \[m_{\alpha,\beta}=\sup_{\alpha,\beta} \, \alpha - e^{\alpha}f(\beta\, p).\]

\noindent Consider $f(\beta \, p)=\int_{\M} e^{\beta \la p,\L(v) \ra}\, d\mu.$ We know that $e^{\la p,\L(v) \ra}$ is at most $1$ on $\M$ and by the assumption of the Lemma, the maximum of $1$ is attained on a set of measure zero. Therefore $f(\beta p)$ can be made arbitrarily small by taking $\beta$ appropriately large.

It follows that $m_{\alpha,\beta}=+\infty$, since we can take arbitrarily large $\alpha$ and then adjust $\beta$ so that $e^{\alpha}f(\beta p)$ is arbitrarily small. Thus $f^*(y)=+\infty$ for all $\y \in \partial B$.
\end{proof}

\section{Moment Cones}

Now we apply the geometric machinery we developed to the moment cones.

\subsection{The (compact) truncated $\K$-moment problem.}
Let $V:=\rd$ denote the vector space of polynomials in $n$ variables of degree at most $2d$, and let $\K \subset \R^n$ be a compact set that is the closure of an open subset of $\R^n$. For this application the set $\M$ of Section \ref{prelim} will simply be $\K$. Let $\mu$ be a measure, supported on $\K$, that is absolutely continuous with respect to the restriction of the Lebesque measure to $\K$.

Let $\p_{2d}(\K)$ denote the cone of polynomials in $\rd$ that are non-negative on $\K$. The cone $\p_{2d}(\K)$ is a closed, convex, full-dimensional, pointed cone in $\rd$. Let $\r_{2d}(\K)\subset \rd^*$ denote the cone of linear functionals on $\rd$ that come from integration with respect to a finite Borel measure supported on $\K$, i.e. the set of all linear functionals $\ell \in \rd^*$ which can be written in the form
$$\ell(p)=\int_{\K}p\,d\sigma \qquad \text{for some measure} \,\, \sigma.$$

For every $v\in \R^n$, let $\ell_v\in\rd^*$ denote the linear functional given by evaluation at the point $v$: $\ell_v(p)=p(v)$. We can view $\ell_v$ as the integrational functional with respect to the Dirac-$\delta$ measure on $v$. Let $\L(\K)\subset\r_{2d}(\K)$ denote the set of all linear functionals $\ell_v\in\rd^*$ with $v \in \K$. Clearly, the set $\L(\K)$ is a continuous embedding of $\K$ into $\rd^*$ taking $v \in \K$ to $\ell_v$. The function $f$ in \eqref{def-f} now reads
\[p\mapsto f(p)\,:=\,\int_{\K}e^{\langle \ell_v,p\rangle}\,d\mu(v)\,=\,\int_\K\ e^{p(v)}\,d\mu(v),\qquad p\in\R[x]_{2d}.\]

Let $H \subset \rd^*$ be the affine hyperplane of all linear functionals that evaluate to $1$ on the constant polynomial $1$:
$$H=\{\ell \in \rd^* \, \mid \, \ell(1)=1\}.$$
It is clear that $\L(\K)$ is contained in $H$. Finally, in order to apply our framework we claim that $\r_{2d}(\K)$ is the conical hull of $\L(\K)$.

\begin{lemma}\label{LEMMA Cone Def}
The cone $\r_{2d}(\K)$ of all linear functionals representable by a measure supported on $\K$ is the conical hull of the set $\L(\K)$ of linear functionals $\ell_v$ with $v \in \K$.
\end{lemma}
\begin{proof}
Let $C(\K)$ denote the conical hull of $\L(\K)$. We note that since $\L(\K)$ is compact and included in the hyperplane
$H$ it follows that the cone $C(\K)$ is closed. The inclusion $C(\K) \subseteq \r_{2d}(\K)$ is straightforward. To prove the reverse inclusion note that the dual cone $C(\K)^*$ of $C(\K)$ is the cone $\p_{2d}(\K)$ of polynomials non-negative on $\K$, simply because the functionals $\ell_v$ with $v \in \K$ encode non-negativity on $\K$. By bi-duality it follows that
$C(\K)=\p_{2d}^*(\K)$.

For every $\ell \in \r_{2d}(\K)$, $\ell(p) \geq 0$ for all $p \in \p_{2d}(\K)$ and therefore
$\r_{2d}(\K) \subseteq \p_{2d}^*(\K)$. Thus we obtain
$$C(\K)\,\subseteq\,\r_{2d}(\K)\,\subseteq\,\p_{2d}^*(\K)\,=\,C(\K),$$
which yields the desired result.
\end{proof}
Lemma \ref{LEMMA Cone Def} can also be derived from extensions of Tchakaloff's theorem 
in Putinar \cite{put-tchaka} and Bayer and Teichman \cite{tchaka}. 

Now we can directly apply Theorem \ref{THEOREM Gradient Image} to $\r_{2d}(\K)$ with $\M=\K$. It still remains to check that the assumptions of Lemma \ref{LEMMA Assumption} hold for $\L(\K)$. Let $F$ be a maximal (by inclusion) face of $\r_{2d}(\K)$. Maximal faces of a convex cone are exposed. Therefore, there exists a form $p \in \p_{2d}(\K)$ such that $\ell(p)=0$ for all $\ell \in F$ and $\ell(p) >0$ for all $\ell \in\r_{2d}(\K) \setminus F$. Now suppose that $\L(v)=\ell_v \in F$. It follows that $\ell_v(p)=p(v)=0$. Therefore the set of $v \in \K$ for which $\L(v)$ is in $F$ corresponds precisely to the zeroes of $p$ in $\K$:
$$\{v \in \K \, \mid \, \L(v) \in F\}=\{ v \in \K \, \mid \, p(v)=0\}.$$
Since $K$ is the closure of an open set in $\R^n$ and the measure $\mu$ is absolutely continuous with respect to the Lebesque measure it follows that $\mu(\{v \in \K \, \mid \, \L(v) \in F\})=0$.

In summary we have proved the following result:

\begin{thm}
\label{thcompact}
Let $\K\subset\R^n$ be the closure of an open bounded subset.
Let $\mu$ be an arbitrary finite Borel measure on $\K$, absolutely continuous
with respect to the restriction of the Lebesgue measure on $\K$, with a density positive on $\K$, and let
\begin{equation}
\label{thcompact-1}
p\mapsto f(p)\,:=\,\int_\K e^{p(x)}\,d\mu(x),\qquad p\in\R[x]_{2d}.
\end{equation}
A sequence $\y\in\rd^*$ belongs to ${\rm int}\,\r_{2d}(\K)$ if and only if
\begin{equation}
\label{thcompact-2}
f^*(\y)\,:=\,\sup_p\:\left\{ \langle p,\y\rangle-f(p)\:\right\}\,<\,+\infty.
\end{equation}
In other words, ${\rm int}\,\r_{2d}(\K)$ is the domain of the Legendre-Fenchel transform of $f$.
\end{thm}
Observe that the function $f^*$ (resp. $\log f^*$) provides a barrier (resp. log-barrier) for the convex cone $\r_{2d}(\K)$.

Interestingly, the function $\log f$ is well-known to statisticians.  It is called the {\it log partition} (or, {\it cumulant}) function
associated with the so-called {\it potential functions} (or, {\it sufficient statistics}) $(x^\alpha)$, $\alpha\in\N^n_{2d}$.
Rephrased in our context, Theorem 3.3 in Wainwright and Jordan \cite{jordan}
 states that under weak hypotheses,
the mapping $\nabla (\log f)$ is onto the open convex set of moment sequences $\y\in\R[x]^*_{2d}$ that have a representing
measure $\nu$ absolutely continuous with respect to $\mu$ and total mass $1$.
Our result is an extension of \cite[Theorem 3.3]{jordan} as  we prove that
$\nabla f$ is onto ${\rm int}\,\r_{2d}(\K)$, i.e., the interior of the cone of sequences that have {\it arbitrary}
representing measures (as opposed to measures absolutely continuous w.r.t. $\mu$). Furthermore, we guarantee that the Legendre-Fenchel transform $f^*$ is finite
only on the interior of the cone of representable moment sequences, while \cite[Theorem 3.3]{jordan} makes no guarantees of the behavior of $(\log f)^*$ on the boundary.

\subsection{The truncated moment problem on $\R^n$.}
We now consider the more delicate case when $\K=\R^n$. In this case,
as $\K$ is not compact the above machinery cannot be applied directly with $\M=\K$ and a detour is needed.

Let $\r_{2d}(\R^n)\subset \rd^*$ denote the cone of linear functionals on $\rd$ that come from integration with respect to a finite, Borel measure supported on $\R^n$, i.e. the set of all linear functionals $\ell \in \rd^*$ which can be written in the form
$$\ell(p)=\int_{\R^n}p\,d\sigma \qquad \text{for some measure} \,\, \sigma.$$

Unlike the situation of the compact support $\K$, the cone $\r_{2d}(\R^n)$ is no longer closed. However, there is a nice way to represent the closure of $\r_{2d}(\R^n)$. It is well-known that the dual cone of $\r_{2d}(\R^n)$ is the cone $\p_{2d}(\R^n)$ of polynomials nonnegative on all of $\R^n$. By bi-duality, the closure of $\r_{2d}(\R^n)$ is the cone $\p_{2d}^*(\R^n)$, i.e.:
$$\p_{2d}^*(\R^n)=\overline{\r_{2d}(\R^n)}.$$

Given an arbitrary polynomial $p \in \rd$ we can homogenize $p$ by adding an extra variable $x_0$ and multiplying all monomials in $p$ by an appropriate power of $x_0$ so that all monomials have degree $2d$. Let $\overline{p}$ denote the homogenization of $p$. Conversely, we can de-homogenize $\overline{p}$ by setting $x_0=1$ to obtain $p$. If $p$ is a nonnegative polynomial, then $\overline{p}$ is a nonnegative form.

Let $V=\H_{n,2d}$ denote the vector space of all homogeneous forms in $n+1$ variables of degree $2d$. We can linearly identify $\rd$ with $\H_{n,2d}$ via homogenization. Let $\H\p_{2d}$ denote the cone of nonnegative forms on $\R^{n+1}$. From the above it follows that homogenization identifies the cone $\p_{2d}(\R^n)$ with the cone $\H\p_{2d}$.

Define $\H\r_{2d} \subset \H^*_{n,2d}$ to be the cone of linear functionals on $\H_{n,2d}$ given by integration with respect to a finite Borel measure on $\R^{n+1}$. For this example the compact set
$\M$ of Section \ref{prelim} will be the unit sphere $\s^n$ in $\R^{n+1}$. For $v \in \R^{n+1}$ let $\ell_v \in \H_{n,2d}^*$ be the linear functional given by evaluation at $v$, i.e.: $$p\mapsto \ell_v(p)=p(v) \qquad \text{for all} \,\, p \in \H_{n,2d}.$$

As before, let $\L(\M)$ be the set of linear functionals $\ell_v$ with $v \in \M\,(=\s^n)$. Let $H$ be the hyperplane in $\H_{n,2d}^*$ consisting of all functionals that evaluate to $1$ on $(x_0^2+\ldots+x_n^2)^d$. It is clear that $\L(\M)$ is a continuous embedding of $\s^n$ into $\H_{n,2d}^*$ and $\L(\M)$ lies in $H$. The analogue of the function $f$ in (\ref{def-f}) now reads
\[p\mapsto f(p)\,=\,\int_{\s^n}e^{\langle\ell_v,p\rangle}\,d\mu(v)\,=\,
\int_{\s^n}e^{p(v)}\,d\mu(v),\qquad p\in \H_{n,2d}.\]

We next show in the following Lemma that $\H\r_{2d}$ is a closed convex cone and in fact $\H\r_{2d}$ is the conical hull of $\L(\M)$. This is very similar to the situation in Lemma \ref{LEMMA Cone Def} and the proof is almost identical.

\begin{lemma}\label{LEMMA Cone Def1}
The cone $\H\r_{2d}$ of all linear functionals representable by a measure supported on $\R^n$ is the conical hull of the set $\L(\M)$ of linear functionals $\ell_v$ with $v \in \M=\s^n$.
\end{lemma}
\begin{proof}
Let $C(\M)$ denote the conical hull of $\L(\M)$. We note that since $\L(\M)$ is compact and included in the hyperplane
$H$ it follows that the cone $C(\M)$ is closed. The inclusion $C(\M) \subseteq \H\r_{2d}$ is straightforward. To prove the reverse inclusion note that the dual cone $C(\M)^*$ of $C(\M)$ is the cone $\H\p_{2d}$ of non-negative forms, simply because the functionals $\ell_v$ with $v \in \s^n$ encode non-negativity on $\s^n$ and by homogeneity on all of $\R^{n+1}$. By bi-duality it follows that
$C(\M)=\p_{2d}^*(\M)$.

For every $\ell \in \H\r_{2d}$, $\ell(p) \geq 0$ for all $p \in \H\p_{2d}$ and therefore
$\H\r_{2d} \subseteq \H\p_{2d}^*$. Thus we obtain
$$C(\M)\,\subseteq\,\H\r_{2d}\,\subseteq\,\H\p_{2d}^*\,=\,C(\M),$$
which yields the desired result.
\end{proof}

As we have seen in the proof of Lemma \ref{LEMMA Cone Def1}, the cone $\H\r_{2d}$ is dual to the cone $\H\p_{2d}$ of nonnegative forms. Via de-homogenization we can identify $\H\p_{2d}$ with the cone of nonnegative polynomials $\p_{2d}(\r^n)$. Therefore it follows that
$$\H\r_{2d}=\overline{\r_{2d}(\R^n)}=\p_{2d}^*(\R^n)=\H\p^*_{2d}.$$

Now we can apply Theorem \ref{THEOREM Gradient Image} to the cone $\H\r_{2d}$. As before, we need to check that the assumptions of Lemma \ref{LEMMA Assumption} hold for $\L(\M)$. Let $F$ be a maximal (by inclusion) face of $\H\r_{2d}$. Maximal faces of a convex cone are exposed. Therefore, there exists a form $p \in \H\p_{2d}$ such that $\ell(p)=0$ for all $\ell \in F$ and $\ell(p) >0$ for all $\ell \in \H\r_{2d} \setminus F$. Now suppose that $\L(v)=\ell_v \in F$. It follows that $\ell_v(p)=p(v)=0$. Therefore the set of $v \in \s^n$ for which $\L(v)$ is in $F$ corresponds precisely to the zeroes of $p$ in $\s^n$:
$$\{v \in \s^n \, \mid \, \L(v) \in F\}=\{ v \in \s^n \, \mid \, p(v)=0\}.$$
Since the measure $\mu$ is the rotation invariant probability measure on $\s^n$ it follows that $\mu(\{v \in \s^n \, \mid \, \L(v) \in F\})=0$.

In summary we have proved the following Theorem:
\begin{thm}
\label{thRn}
Let $\mu$ be the rotation invariant probability measure on the unit sphere $\s^n\subset\R^{n+1}$, and let
\begin{equation}
\label{thcompact-1Rn}
p\mapsto f(p)\,:=\,\int_{\s^n} e^{p(x)}\,d\mu(x),\qquad p\in \H_{n,2d}.
\end{equation}
A sequence $\y\in \H^*_{n,2d}$ belongs to ${\rm int}\,\H\r_{2d}$ if and only if
\begin{equation}
\label{thcompact-2Rn}
f^*(\y)\,:=\,\sup_p\:\left\{ \langle p,\y\rangle-f(p)\:\right\}\,<\,+\infty.
\end{equation}
In other words, ${\rm int}\,\H\r_{2d}$ is the domain of the Legendre-Fenchel transform of $f$.
\end{thm}
Observe that the function $f^*$ (resp. $\log f^*$) provides a barrier (resp. log-barrier) for the convex cone $\H\r_{2d}$.\\

Finally, to relate the initial moment problem in $\R^n$ with
the one in $\R^{n+1}$ with homogenization, observe that a sequence $\y=(y_\alpha)\in \R[x_1,\ldots,x_n]_{2d}^*$
has a representing measure on $\R^n$ if and only if the sequence $\tilde{\y}\in\H^*_{n,2d}$ defined by:
\[\tilde{y}_{\alpha k}\,=\,y_\alpha,\qquad \forall (\alpha,k)\in\N^{n+1}_{2d},\quad \vert\alpha\vert+k=2d,\]
has a representing measure on $\R^{n+1}$. 
Indeed, $\y$ has a representing measure $\mu$ on $\R^n$ if and only if $\tilde{\y}$ has the representing (product) measure
$\mu\otimes \delta_{x_0=1}$ on $\R^{n+1}$ (where $\delta_{x_0=1}$ is the Dirac measure at $x_0=1$).

\subsection{General Non-Compact Case}\label{SUBSECTION Non-Compact}
With identical proofs via homogenization, the above discussion of the case $\K=\R^n$
can be extended to any set $\K$ which  is the closure an open subset of $\R^n$. We explain how to define the appropriate compact set $\M \subset \s^n$.

Embed $\K$ into $\R^{n+1}$ by introducing an extra coordinate and setting it equal to $1$. More formally, let $\K' \subset \R^{n+1}$ given by:
\[\K'=\left\{(x,1) \in \R^{n+1} \,\, \mid \,\, x \in \K\right\}.\]

\noindent Next, define, $\M' \subset \s^n$ as a rescaling of $\K'$ onto the unit sphere:

\[\M'=\left\{z \in \s^{n} \,\, \mid \,\, z\,=\,\lambda y \,\, \text{for some}\,\, \lambda \in \R, y \in \K'\right\}.\]

\noindent Finally, let $\M \subset \s^n$ be the closure of $\M'$: $\M=\overline{\M}'$. Take any finite measure $\mu$ on $\M$ that is absolutely continuous with respect to the rotation invariant probability measure on the unit sphere $\s^n$. This makes sense, since $\K$ is the closure of an open subset of $\R^n$, and therefore $\M$ is a closure of an open subset of $\s^n$. We note that with $\K=\R^n$ the above construction gives $\M=\s^n$.

At last, nonnegativity of polynomials on $\K$ is equivalent to nonnegativity of forms on $\M$. The rest of the proofs follow nearly word for word.
\section{Proofs of Lemmas}\label{SECTION Lemma Proofs}
\begin{proof}[Proof of Lemma \ref{LEMMA Convex}]
We observe that $\overline{A}$ ia a closed compact set. Since exposed extreme points are dense in the set of extreme points, it follows that $\overline{A}$ contains all extreme points of $B$ and by Krein-Milman Theorem the condition of the lemma is equivalent to $\overline{A}=B$.

Now suppose that there exists $x \in \operatorname{int} B$ such that $x \notin A$. Then by the Separation Theorem, there exists a hyperplane $H$ such that $x$ lies in the closed half-space $\overline{H}_+$ and $A$ lies in the closed half-space $\overline{H}_-$.

Since $x \in \operatorname{int} B$ it follows that there exist points $y \in B$, such that $y$ lies in the open half-space $H_-$. Since $A \subset \overline{H}_+$ we see that $\overline{A}$ cannot contain such points $y$, which is a contradiction.
\end{proof}
\begin{proof}[Proof of Lemma \ref{LEMMA Dirac}]

First we note that by dividing through by the maximum of $f$ we may restrict 
ourselves to the case $0<f(x)\leq 1$ for all $x \in \M$. Next, for every $\lambda\in\N$, define the probability measure
\[\nu_\lambda(B)\,:=\,\left(\displaystyle\int_{\M}f^\lambda \,d\mu\right)^{-1}\displaystyle\int_B f^\lambda\,d\mu,\qquad B\in\mathcal{B}(\M).\]

Fix $\epsilon>0$ and let $\M(\epsilon)\subset\M$ be the open set $\{x\in\M\,:\,f(x)<1-\epsilon\}$.
As $\mu$ has a density positive on $\M$,
\begin{eqnarray}
\label{aux1}
\lim_{\lambda\to\infty}\:\left(\int_{\M(\epsilon)}f^\lambda\,d\mu\right)^{1/\lambda}&=&
{\rm esssup}\,\{f(x)\,:\,x\in\M(\epsilon)\}\,=\,(1-\epsilon)\\
\label{aux2}
\lim_{\lambda\to\infty}\:\left(\int_{\M}f^\lambda\,d\mu\right)^{1/\lambda}&=&
{\rm esssup}\,\{f(x)\,:\,x\in\M\}\,=\,1.
\end{eqnarray}
Therefore,
\[(1-\epsilon)\,=\,\lim_{\lambda\to\infty}\frac{\left(\displaystyle\int_{\M(\epsilon)}f^\lambda\,d\mu\right)^{1/\lambda}}{\left(\displaystyle\int_{\M}f^\lambda\,d\mu\right)^{1/\lambda}}\,=\,\lim_{\lambda\to\infty}\nu_\lambda(\M(\epsilon))^{1/\lambda},\]
which in turn implies
$\lim_{\lambda\to\infty}\nu_\lambda(\M(\epsilon))=0$ for any fixed $\epsilon >0$. Therefore we also have $\lim_{\lambda\to\infty}\nu_\lambda(\M \setminus \M(\epsilon))=1$. Let us evaluate
\[\overline{\rho}\,:=\,\limsup_{\lambda\to\infty}\int_\M g\,d\nu_\lambda; \quad\underline{\rho}\,:=\,\liminf_{\lambda\to\infty}\int_\M g\,d\nu_\lambda.\]
Let $(\lambda_j)$, $j\in\N$, be a subsequence such that
\[\overline{\rho}:=\lim_{j\to\infty}\int_{\M} g\,d\nu_{\lambda_j},\]
and write
\[\overline{\rho}:=\lim_{j\to\infty}\left(\underbrace{\int_{\M(\epsilon)} g\,d\nu_{\lambda_j}}_{A^\epsilon_j}+
\underbrace{\int_{\M\setminus\M(\epsilon)} g\,d\nu_{\lambda_j}}_{B^\epsilon_j}\right)
\,=\,\lim_{j\to\infty}\left(A^\epsilon_j+B^\epsilon_j\right).\]
Observe that 
\[\vert A^\epsilon_j\vert\,\leq\,\sup_{x\in\M}\vert g(x)\vert\: \nu_{\lambda_j}(\M(\epsilon)),\qquad\forall\epsilon>0,\]
so that $\lim_{j\to\infty}A^\epsilon_j=0$. Therefore,
\[\overline{\rho}\,=\,\lim_{j\to\infty}\,B^\epsilon_j\,=\,\lim_{j\to\infty}\int_{\M\setminus\M(\epsilon)} g\,d\nu_{\lambda_j}.\]
Using  the Mean Value theorem,
\[\overline{\rho}\,=\,\lim_{j\to\infty}\,B^\epsilon_j\,=\,\lim_{j\to\infty}g(\xi_j)\,\nu_{\lambda_j}(\M\setminus\M(\epsilon)),\]
for some $\xi_j\in\M\setminus\M(\epsilon)$. Since $\lim_{\lambda\to\infty}\nu_\lambda(\M \setminus \M(\epsilon))=1$ we see that $$\overline{\rho}\,=\,\lim_{j\to\infty}g(\xi_j).$$
Next, we also have
\[(1-\epsilon')g(s)\,\leq\,g(x)\,\leq\,(1+\epsilon')g(s),\qquad\forall x\in\M\setminus\M(\epsilon),\]
with $\epsilon'\to0$ as $\epsilon\to0$. And so 
$\vert g(\xi_j)-g(s)\vert \leq \epsilon'g(s)$ which implies 
$$-\epsilon ' g(s)\leq\overline{\rho}-g(s)\leq \epsilon'g(s).$$

Exactly same arguments for the subsequence $(\lambda_k)$, $k\in\N$, such that
\[\underline{\rho}\,=\,\liminf_{\lambda\to\infty}\int_\M g\,d\nu_\lambda\,=\,\lim_{k\to\infty}\int_\M g\,d\nu_{\lambda_k},\]
may be used to prove $-\epsilon 'g(s)\leq\underline{\rho}-g(s)\leq\epsilon'g(s)$. Letting $\epsilon\to0$
(hence $\epsilon'\to0$) yields the desired result (\ref{result}).
\end{proof}

\section*{Acknowledgement}
The first author was partially supported by NSF and wishes to acknowledge financial support from the University of Toulouse, France,
for a visit to LAAS in June-July 2011, during which this research was performed.


\begin{thebibliography}{}
\bibitem{tchaka}
C. Bayer, J. Teichmann. The proof of Tchakaloff's theorem, Proc. Amer. Math. Soc. \textbf{134} (2006), pp. 3035--3040.
\bibitem{berg}
C. Berg. The multidimensional moment problem and semi-groups, Proc. Symp.
Appl. Math. \textbf{37} (1987), pp. 110--124.
\bibitem{borwein}
J.M. Borwein and A.S. Lewis. On the convergence of moment problems, Trans. Amer. Math. Soc. \textbf{325} (1991), pp. 249--271.
\bibitem{curto0}
R.E. Curto and L.A. Fialkow. Recursiveness, positivity and truncated moment problems,
Houston J. Math. \textbf{17} (1991), pp. 603--635.
\bibitem{curto1}
R.E. Curto and L.A. Fialkow. The truncated complex K-moment problem,
Trans. Amer. Math. Soc. \textbf{352} (2000), pp. 2825--2855.
\bibitem{curto2}
R.E. Curto and L.A. Fialkow. Truncated K-moment problems in several variables,
J. Oper. Theory \textbf{54} (2005), pp. 189--226.
\bibitem{krivine1}
J. L. Krivine.  Anneaux pr\'eordonn\'es, J. Analyse Math. \textbf{12} (1964), pp. 307--326.
\bibitem{krivine2}
J. L. Krivine. Quelques propri\'et\'es des pr\'eordres dans les anneaux commutatifs unitaires, C.R.
Acad. Sci. Paris \textbf{258} (1964), pp. 3417--3418.
\bibitem{marshall}
M. Marshall. A general representation theorem for partially ordered commutative rings,
Math. Z. \textbf{242} (2002), pp. 217--225.
\bibitem{put-tchaka}
M. Putinar. A note on Tchakaloff's theorem, Proc. Amer. Math. Soc. \textbf{125} (1997), pp. 2409--2414. 
\bibitem{putinar}
M. Putinar. Positive polynomials on compact semi-algebraic sets,
Indiana Univ. Math. J. \textbf{42} (1993), pp. 969--984.
\bibitem{putval}
M. Putinar and F.-H. Vasilescu. Solving the moment problem by dimensional
extension, Ann. Math. \textbf{149} (1999), pp. 1087--1107.
\bibitem{roc}R.~T.~Rockafellar: {\em{Convex Analysis}}, Princeton University Press, 1970.
\bibitem{sarason}
D. Sarason. Moment problems and operators in Hilbert space,
Proc. Symp. Appl. Math. \textbf{37} (1987), pp. 54--70.
\bibitem{schmudgen}
K. Schm\"udgen. The K-moment problem for compact semi-algebraic sets, Math. Ann.
\textbf{289} (1991), pp. 203--206.
\bibitem{simon}
B. Simon. The classical moment problem as a self-adjoint finite difference operator,
Adv. Math. \textbf{137} (1998), pp. 82--203.
\bibitem{vasilescu}
F.-H. Vasilescu. Spectral measures and moment problems,
Spectral Theory and Its Applications, Theta 2003, pp. 173--215.
\bibitem{jordan}
M.J. Wainwright and M.I. Jordan. Graphical Models, Exponential Families, and Variational Inference,
Foundations and Trends in Machine Learning \textbf{1:1-2} (2008).
\bibitem{integration}
M. Wester, Y. Yaacob, S. Steinberg. Computing integrals over polynomially defined regions and their boundaries in 2 and 3 dimensions,
Math. Comp. Simul. \textbf{82} (2011), pp. 79--101.
\end{thebibliography}
\end{document}